\documentclass[twoside]{article}

\usepackage{amssymb,amsfonts,amssymb,amsthm,amsmath}
\usepackage{mathrsfs}                      
\usepackage{graphicx}                      
\usepackage[latin1]{inputenc}             
\usepackage{color}

\begin{document}

\title{Continuity of attractors for a nonlinear parabolic problem with terms concentrating in the boundary}

\author{Gleiciane S. Arag\~ao\thanks{Universidade Federal de S\~ao Paulo, UNIFESP, Diadema, Brazil, e-mail: gleiciane.aragao@unifesp.br, Partially
supported by FAPESP 2010/51829-7, Brazil.}\,\,,\, Ant\^onio L. Pereira\thanks{Instituto de Matem\'atica e Estat\'istica, USP, S\~ao Paulo, Brazil, e-mail: alpereir@ime.usp.br, Partially supported by CNPq 308696/2006-9, FAPESP 2008/55516-3, Brazil.}\,\, and\, Marcone C. Pereira\thanks{Escola de Artes, Ci\^encias e Humanidades, USP, S\~ao Paulo, Brazil, e-mail: marcone@usp.br, Partially supported by CNPq 305210/2008-4 and 302847/2011-1, FAPESP 2008/53094-4 and 2010/18790-0, Brazil.}}
\date{}
\maketitle \thispagestyle{empty} \vspace{-10pt}

\begin{abstract}
We analyze the dynamics of the flow generated  by a nonlinear parabolic problem when some reaction and potential terms are concentrated in a neighborhood of the boundary. We assume that this neighborhood shrinks to the boundary as a parameter $\epsilon$ goes to zero. Also, we suppose that the ``inner boundary'' of this neighborhood presents a highly oscillatory behavior. 
Our main goal here is to show the continuity of the family of attractors with respect to $\epsilon$. Indeed, we prove upper semicontinuity under the usual properties of regularity and  dissipativeness and, assuming hyperbolicity of the equilibria, we also show  the lower semicontinuity of the attractors at $\epsilon=0$.  

\vspace{0.2cm}


\end{abstract}

\section{Introduction}
\label{introd}

\ \ \ \ Let $\Omega \subset \mathbb{R}^2$ be an open bounded set with a $C^2$-boundary $\partial \Omega$ and $g_\epsilon(\cdot)$ a function satisfying $0 < g_0 \leqslant g_\epsilon(\cdot) \leqslant g_1$ for fixed positive constants $g_0$ and $g_1$, which may oscillate as the small parameter $\epsilon \to 0$. 
This is expressed by  
$$
g_\epsilon(s) =g(s,s/\epsilon),
$$
where the function $g : (0, T) \times \mathbb{R} \mapsto \mathbb{R}$, $T>0$, is a positive smooth function such that $y \to g(x, y)$ is $l(x)$-periodic in $y$ for each $x$, with period $l(x)$ uniformly bounded in $(0,T)$, that is, $0<l_0<l(\cdot)<l_1$.

Also, let $x,y \in C^{2}([0,T])$ such that the curve $\zeta(s)=(x(s),y(s))$, $s\in[0,T]$, is a $C^{2}$-parametrization of the boundary $\partial \Omega$ with $\left\|\zeta'(s)\right\|=1$, for all $s\in [0,T]$. We also assume  that $N(\zeta(s))=(y'(s),-x'(s))$ is the unit outward normal vector to $\partial \Omega$, and  we define the $\epsilon$-strip neighborhood for the boundary $\partial \Omega$ by  
$$
\omega_{\epsilon}=\left\{ \xi\in \mathbb{R}^{2} \; : \; \xi= \zeta(s)-tN(\zeta(s)), \quad s\in [0,T] \quad \mbox{and} \quad 0\leqslant t<\epsilon \, g_{\epsilon}(s) \right\}, 
$$
for $\epsilon>0$ sufficiently small, say $0<\epsilon\leqslant\epsilon_{0}$. 

Observe that our assumptions include the case where the oscillating function $g_\epsilon$ presents a purely periodic behavior as, for instance,  $g_\epsilon(s) = 2 + cos(s/\epsilon)$, but also contain the case where $g_\epsilon$ is not periodic and the amplitude is modulated by a function. 
For small $\epsilon$, the set $\omega_{\epsilon}$ is a neighborhood of $\partial \Omega$ in $\bar{\Omega}$, that collapses to the boundary when the parameter $\epsilon$ goes to zero. Note that the ``inner boundary'' of $\omega_{\epsilon}$,
$
\left\{ \xi\in \mathbb{R}^{2} \; : \; \xi= \zeta(s)- g_{\epsilon}(s) N(\zeta(s)), \quad s\in [0,T]\right\}, 
$  presents a highly oscillatory behavior. Moreover, the height of $\omega_{\epsilon}$, the amplitude and period of the oscillations are all of the same order, given by the small parameter $\epsilon$. 
See Figure 1 that illustrates the oscillating strip $\omega_{\epsilon}$ for the purely periodic case.

\begin{figure}[!h]
  \centering
  \includegraphics[width=.40\columnwidth]{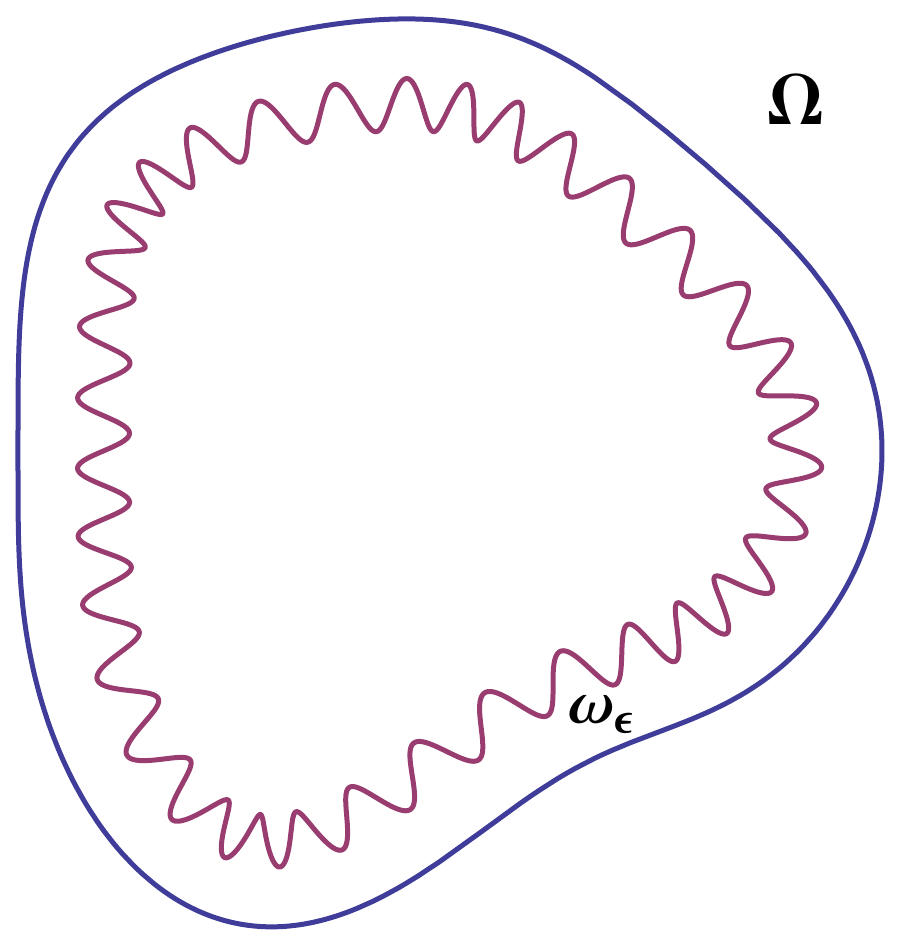} 
  \caption{The open set $\Omega$ and the oscillating strip $\omega_\epsilon$.}
  \label{figfuncaog} 
\end{figure}

In this work we are interested in the asymptotic behavior of the solutions of the nonlinear parabolic problem
\begin{eqnarray}
\label{1}
\left\{ \begin{array}{lll}
         \frac{\partial u_\epsilon}{\partial t}-\Delta u_\epsilon+\lambda u_\epsilon+\frac{1}{\epsilon}\mathcal{X}_{\omega_{\epsilon}}V_{\epsilon}u_\epsilon= \frac{1}{\epsilon}\mathcal{X}_{\omega_{\epsilon}}f(u_\epsilon) &   \mbox{ in } \Omega, \; t > 0,\\
         \frac{\partial u_\epsilon}{\partial N}=0  & \mbox{ on } \partial\Omega, \; t>0,\\
         u_\epsilon(0)=\phi^{\epsilon}\in H^{1}(\Omega),
\end{array} \right.
\end{eqnarray}
as $\epsilon>0$ goes to zero.
$\mathcal{X}_{\omega_{\epsilon}}$ is the characteristic function of the set $\omega_{\epsilon}$, $\lambda$ is a suitable real number and  the nonlinearity $f:\mathbb{R} \mapsto \mathbb{R}$ is a $C^{2}$-function. 
We assume that there exists $C>0$ independent of $\epsilon$ such that the family of potential $V_{\epsilon} \in L^\infty(\Omega)$ satisfies
\begin{eqnarray}
\label{hip1}
\frac{1}{\epsilon} \int_{\omega_{\epsilon}} \left|V_{\epsilon}(x,y) \right|^{2} \, dx dy \leqslant C.
\end{eqnarray} 
Also, we suppose there exists a function $V_{0}\in L^{2}(\partial \Omega)$ which is the weak  limit of the concentrating term 
\begin{eqnarray}
\label{hip2}
\lim_{\epsilon\to 0} \frac{1}{\epsilon} \int_{\omega_{\epsilon}} V_{\epsilon} \, \varphi \, d\xi = \int_{\partial \Omega}V_{0} \, \varphi \, dS, \qquad \mbox{$\forall$ $\varphi \in C^{\infty}(\bar{\Omega})$}.
\end{eqnarray}    

We are using here the  characteristic functions $\mathcal{X}_{\omega_{\epsilon}}$ depending on a small positive parameter $\epsilon$ modeling the concentration on the region $\omega_\epsilon \subset \bar \Omega$ through the term 
$$
\frac{1}{\epsilon} \mathcal{X}_{\omega_{\epsilon}} \in L^\infty(\Omega).
$$
Roughly, we are assuming that the reactions of the problem \eqref{1} occur only in an extremely oscillating thin region near the border. 
Furthermore, we also allow potential terms concentrating in this strip. 
In some sense, we will prove that this singular problem can be approximated by a parabolic problem with nonlinear boundary conditions, where the oscillatory behavior of the neighborhood is captured as a flux condition and a potential term on the boundary.

It is reasonable to expect that the family of solutions $u_\epsilon$ will converge to a
solution of  an equation of the same type with nonlinear boundary condition on $\partial \Omega$ since $\omega_\epsilon$ is a thin strip ``approaching''  $\partial \Omega$.
Indeed, we show that 
under certain conditions, the limit problem of (\ref{1}) is the following parabolic problem with nonlinear boundary conditions  
\begin{eqnarray}
\label{2}
\left\{ \begin{array}{lll}
         \frac{\partial u_0}{\partial t}-\Delta u_0+\lambda u_0=0 & \mbox{ in } \Omega, \; t>0,\\
         \frac{\partial u_0}{\partial N}+V_{0}u_0=\mu \,  f(u_0) & \mbox{ on } \partial \Omega, \; t>0,\\
         u_0(0)=\phi^{0}\in H^{1}(\Omega),
        \end{array} \right.
\end{eqnarray}
where the boundary coefficient $\mu \in L^\infty( \partial \Omega)$ is related to the oscillating function $g_\epsilon$ and is given by 
\begin{equation} \label{mu}
\mu(s) = \mu(\zeta(s))=\frac{1}{l(s)} \int_0^{l(s)} g(s,\tau) \, d\tau, \qquad \mbox{$\forall$ } s \in (0,T).
\end{equation}
As mentioned, we get a limit problem with a nonlinear boundary condition that captures the oscillatory behavior of the ``inner boundary'' of the set $\omega_{\epsilon}$. This nonlinear boundary condition  includes the function $\mu(s)$, the mean value of $g(s,\cdot)$ for each $s \in (0,T)$.

We are interested in the behavior of the attractors of (\ref{1}) and (\ref{2}) for small $\epsilon>0$. We will show that they are continuous at $\epsilon=0$. 
Recall that	an attractor is a compact invariant set which attracts all bounded sets of the phase space of a dynamical system.


This kind of problem was initially studied in~\cite{arrieta}, where linear elliptic equations were considered. There, the neighborhood is a strip of width $\epsilon$ and base in a portion of the boundary, without oscillatory behavior. Later, the asymptotic behavior  of a  parabolic problem of the same type was analyzed in~\cite{anibal,anibal2}, where the upper semicontinuity of attractors at $\epsilon=0$ was proved. The same technique of~\cite{arrieta} has been used in~\cite{gleicesergio1,gleicesergio2}, where the results of~\cite{arrieta,anibal} were extended to a reaction-diffusion problem with delay. In these works, the boundary of the domain is smooth. 
Recently, in~\cite{gam1}, some results of~\cite{arrieta} were adapted to a nonlinear elliptic problem posed on a Lipschitz domain $\Omega$ presenting a highly oscillatory behavior on the neighborhood of the  boundary using some ideas of~\cite{marcone1,marcone2}, where elliptic and parabolic problems defined in thin domains with a highly oscillatory behavior have been extensively studied.
 
The goal of our work is to extend the results of~\cite{anibal,anibal2} to a  parabolic problem in which the ``inner boundary'' of  $\omega_{\epsilon}$ presents a highly oscillatory behavior. 
Moreover, assuming hyperbolicity of the equilibria of the limit problem,  we also obtain results on the  lower semicontinuity of the attractors. Our approach will be somewhat different from the one in \cite{anibal,anibal2} and closer to the one in \cite{antoniomarcone}, where some abstract results on the continuity of invariant manifolds were  obtained.
 Throughout this work, we  suppose the nonlinearity $f: \mathbb{R} \mapsto \mathbb{R}$ is a $C^{2}$-function satisfying the dissipativeness assumption
\begin{eqnarray}
\label{dissipative}
\limsup_{\left|s\right|\to\infty} \frac{f(s)}{s}<0.
\end{eqnarray}

It has been shown that the parabolic problems (\ref{1}) and (\ref{2}) are well posed in $H^{1}(\Omega)$ and, for each $0\leqslant \epsilon \leqslant \epsilon_{0}$, we have well defined nonlinear semigroup in $H^{1}(\Omega)$ associated to the solutions of (\ref{1}) and (\ref{2}),
$$
T^{\epsilon}(t)\phi = u_\epsilon(t,\phi), \qquad \textrm{ with } t\geqslant0 \textrm{ and } \phi \in H^1(\Omega),
$$  
see for example~\cite{alexandreat,OP}. Moreover, under assumption (\ref{dissipative}), the problems (\ref{1}) and (\ref{2}) have a global attractor $\mathscr{A}_{\epsilon}$, which is bounded in $L^{\infty}(\Omega)$, uniformly in $\epsilon$. 
In particular, if the initial conditions are uniformly bounded, then all solutions of (\ref{1}) and (\ref{2}) are bounded with a bound independent of $\epsilon$. This enables us to cut the nonlinearity $f$ in such a way that it becomes bounded with bounded derivatives up to second order without changing the attractors. Therefore, we may assume without loss of generality that 

\begin{itemize}
\item[(H)] $f: \mathbb{R} \mapsto \mathbb{R}$ is a $C^{2}$-function satisfying (\ref{dissipative}) and
$$
\left|f(u)\right|+\left|f'(u)\right|+\left|f''(u)\right|\leqslant K,\qquad \mbox{$\forall$ $u\in \mathbb{R}$},
$$
for some constant $K>0$. 
\end{itemize}

Although we restrict our attention to nonlinearities independent of the spacial variable, the method can be easily adapted for the case $f=f(x,u)$ depending on $x \in \Omega$. It is worth to mention that we also can consider reactions occurring on the whole region, instead of 
 concentrating on the boundary. In this case, the limit problem would be a non-homogeneous parabolic problem in $\Omega$ with nonlinear boundary conditions.

The paper is organized as follows: in Section~\ref{concentrating}, we describe 
some  technical results, in particular some concerning the concentrating integrals defined in \cite{arrieta}. 
In Section~\ref{abstract}, we introduce an abstract setting to deal the problems (\ref{1}) and (\ref{2}). 
In Section~\ref{continuityeq}, we obtain the upper semicontinuity of attractors at $\epsilon=0$ in $H^{1}(\Omega)$ and prove the continuity of the set of equilibria, assuming that the equilibrium points of (\ref{2}) are hyperbolic. 
In Section~\ref{manifolds}, we show the continuity of the local unstable manifolds near a hyperbolic equilibrium, from which the lower semicontinuity of attractors at $\epsilon=0$ in $H^{1}(\Omega)$ follows.


\section{Concentrating integrals}
\label{concentrating}
 
\ \ \ \ In this section we describe some technical results that will be needed in the sequel. Initially, we adapt some results from \cite{arrieta} on concentrating integrals. We note that since $0 < g_0 \leqslant g_\epsilon(\cdot) \leqslant g_1$ in $(0,T)$, uniformly in $\epsilon$, we have that the set $\omega_{\epsilon}$ is contained in a strip of width $\epsilon g_{1}$ on $\partial \Omega$, without oscillatory behavior. 

\newtheorem{lem1}{Lemma}[section]
\begin{lem1}
\label{lem1}
Suppose that $v\in H^{s}(\Omega)$ with $\frac{1}{2}<s\leqslant 1$ and $s-1\geqslant - \frac{1}{q}$. Then, for sufficiently small $\epsilon_{0}$, there exists a constant $C>0$ independent of $\epsilon$ and $v$ such that for any $0<\epsilon\leqslant\epsilon_{0}$, we have
$$
\frac{1}{\epsilon}\int_{\omega_{\epsilon}}\left|v\right|^{q} d\xi\leqslant C \left\|v\right\|^{q}_{H^{s}(\Omega)}.
$$ 
\end{lem1}
\begin{proof}
We note that $\frac{1}{\epsilon}\int_{\omega_{\epsilon}}\left|v\right|^q d\xi \leqslant \frac{1}{\epsilon}\int_{r_{\epsilon}}\left|v\right|^q d\xi,$ where $r_{\epsilon}$ is given by 
\begin{eqnarray}
\label{repsilon}
r_{\epsilon}=\left\{ \xi\in \mathbb{R}^{2} \; : \; \xi= \zeta(s)-\epsilon g_{1}N(\zeta(s)), \quad s\in [0,T] \right\}.
\end{eqnarray}
Thus, the result follows from~\cite[Lemma 2.1]{arrieta}. 
\end{proof}

In the following result, we describe how our concentrating integrals converge to boundary integrals. 

\newtheorem{lem2}[lem1]{Lemma}
\begin{lem2}
\label{lem2}
Suppose that $h, \varphi \in H^{s}(\Omega)$, with $\frac{1}{2}<s\leqslant 1$. Then,
\begin{eqnarray}
\label{limiteexplicito}
\lim_{\epsilon \to 0}\frac{1}{\epsilon}\int_{\omega_{\epsilon}}h \, \varphi \, d\xi=\int_{\partial \Omega} \mu \, \gamma(h)\gamma(\varphi) dS,
\end{eqnarray}
where $\mu \in L^\infty(\partial \Omega)$ is given by \eqref{mu} and $\gamma: H^{s}(\Omega)\mapsto L^{2}(\partial \Omega)$ is the trace operator.
\end{lem2}
\begin{proof}
Initially, let $h$ and $\varphi$ be smooth functions defined in $\bar{\Omega}$. We note that
$$
\frac{1}{\epsilon}\int_{\omega_{\epsilon}} h \, \varphi \, d\xi=\frac{1}{\epsilon}\int^{T}_{0}\int^{\epsilon g_{\epsilon}(s)}_{0}h(\Psi(t,s))\varphi(\Psi(t,s))\left|\det J \Psi(t,s)\right|dtds,
$$
where $\Psi(t,s)=\zeta(s)-tN(\zeta(s))=(x(s)-ty'(s),y(s)+tx'(s))$, for $s\in[0,T]$ and $0\leqslant t<\epsilon g_{\epsilon}(s)$. Now, for $\epsilon$ sufficiently small, we obtain 
$$
\det J \Psi(t,s)=-[(x'(s))^2+(y'(s))^2]+t(x'(s)y''(s)-y'(s)x''(s))=-1+t(x'(s)y''(s)-y'(s)x''(s))<0.
$$ 
Hence,
$$
\frac{1}{\epsilon}\int_{\omega_{\epsilon}}h\varphi d\xi=\frac{1}{\epsilon}\int^{T}_{0}\int^{\epsilon g_{\epsilon}(s)}_{0}h(\Psi(t,s))\varphi(\Psi(t,s))\left[ 1+t(y'(s)x''(s)-x'(s)y''(s)) \right]dtds.
$$

Taking $t=\epsilon g_{\epsilon}(s) \beta$ we have
$$
\frac{1}{\epsilon}\int^{T}_{0}\int^{\epsilon g_{\epsilon}(s)}_{0}h(\Psi(t,s))\varphi(\Psi(t,s))\left[ 1+t(y'(s)x''(s)-x'(s)y''(s)) \right]dtds
$$
$$
=\int^{T}_{0}\int^{1}_{0}h(\Psi(\epsilon g_{\epsilon}(s)\beta,s))\varphi(\Psi(\epsilon g_{\epsilon}(s)\beta,s))\left[ 1+\epsilon g_{\epsilon}(s)\beta(y'(s)x''(s)-x'(s)y''(s)) \right]g_{\epsilon}(s)d\beta ds.
$$
Thus,
$$
\left|\frac{1}{\epsilon}\int_{\omega_{\epsilon}} h \, \varphi \, d\xi-\int_{\partial \Omega} \mu \, h \, \varphi \, dS\right|\leqslant \left|\int^{T}_{0}g_{\epsilon}(s)h(\gamma(s))\varphi(\gamma(s))ds-\int^{T}_{0}\mu(s) \, h(\gamma(s))\varphi(\gamma(s))ds\right|
$$
$$
+ \displaystyle \left|\int^{T}_{0}\int^{1}_{0} g_{\epsilon}(s) \{ h(\Psi(\epsilon g_{\epsilon}(s)\beta,s))\varphi(\Psi(\epsilon g_{\epsilon}(s)\beta,s))\left[ 1+\epsilon g_{\epsilon}(s)\beta(y'(s)x''(s)-x'(s)y''(s)) \right]-h(\gamma(s))\varphi(\gamma(s))\}d\beta ds\right|.
$$

Using the Average Theorem, we can get the following weak convergence for the oscillating functions $g_\epsilon$
$$
g_\epsilon(\cdot)  \to  \mu(\cdot) = \frac{1}{l(\cdot)}\int_0^{l(\cdot)}g(\cdot,\tau) \,d\tau \quad  w^*-L^\infty(0,T),
$$
for more details see \cite[Lemma 2.3]{gam1}.
Consequently
$$
\lim_{\epsilon\to0}\int^{T}_{0}g_{\epsilon}(s)h(\gamma(s))\varphi(\gamma(s))ds=\int^{T}_{0} \mu(s) \, h(\gamma(s))\varphi(\gamma(s))ds.
$$
Moreover, since $\epsilon g_{\epsilon}(s)\beta \to 0$, as $\epsilon\to0$, uniformly for $(\beta,s)\in [0,1]\times[0,T]$, we have
$$
\displaystyle \left|\int^{T}_{0}\int^{1}_{0} g_{\epsilon}(s) \{ h(\Psi(\epsilon g_{\epsilon}(s)\beta,s))\varphi(\Psi(\epsilon g_{\epsilon}(s)\beta,s))\left[ 1+\epsilon g_{\epsilon}(s)\beta(y'(s)x''(s)-x'(s)y''(s)) \right]-h(\gamma(s))\varphi(\gamma(s))\}d\beta ds\right|\to 0, 
$$
as $\epsilon\to0$. Therefore, $\left|\frac{1}{\epsilon}\int_{\omega_{\epsilon}} h\, \varphi \, d\xi -  \int_{\partial \Omega}\mu \, h \, \varphi \, dS\right|\to 0$ as $\epsilon\to0$. Hence, the proof of equality (\ref{limiteexplicito}) follows from density arguments, the continuity of the trace operator $\gamma$ and Lemma~\ref{lem1}.  
\end{proof}

Also, we obtain the following result as a consequence of Lemma \ref{lem1} and \cite[Lemma 2.5]{arrieta}:
\newtheorem{lem3}[lem1]{Lemma}
\begin{lem3}
\label{lem3}
Suppose that the family $V_{\epsilon}$ satisfies (\ref{hip1}) and (\ref{hip2}). Then, for $s>\frac{1}{2}$, $\sigma>\frac{1}{2}$ and $s+\sigma>\frac{3}{2}$, if we define the operators $P_{\epsilon}: H^{s}(\Omega) \mapsto (H^{\sigma}(\Omega))'$ by
$$
\left\langle P_{\epsilon}(u),\varphi\right\rangle=\frac{1}{\epsilon}\int_{\omega_{\epsilon}}V_{\epsilon} \, u \, \varphi \, d\xi, \quad \textrm { for }  \epsilon>0,   \quad \mbox{ and } \quad  \left\langle P_{0}(u),\varphi\right\rangle=\int_{\partial \Omega} V_{0} \, u \, \varphi \, dS,
$$
we have $P_{\epsilon} \to P_{0}$ in $\mathscr{L}(H^{s}(\Omega),(H^{\sigma}(\Omega))')$.
\end{lem3}

%

\section{Abstract setting}
\label{abstract}

\ \ \ \ We initially proceed as \cite{gam1,arrieta,anibal,anibal2} writing the parabolic problems (\ref{1}) and (\ref{2}) in an abstract form. To this aim,  we introduce the continuous bilinear forms $a_\epsilon: H^1(\Omega) \times H^1(\Omega) \mapsto \mathbb{R}$ with $\epsilon \in [0,\epsilon_0]$ for some $\epsilon_0 > 0$ by
\begin{equation} 
\label{CF}
\begin{gathered}
a_\epsilon(u,v) = \int_\Omega \nabla u \, \nabla v \, dxdy + \lambda \int_\Omega u \, v \, dxdy + \frac{1}{\epsilon} \int_{\omega_{\epsilon}} V_\epsilon \, u \,  v \, dxdy, \qquad   0<\epsilon\leqslant \epsilon_0, \\
a_0(u,v) = \int_\Omega \nabla u \, \nabla v \, dxdy + \lambda \int_\Omega u \, v \, dxdy +  \int_{\partial \Omega} V_0 \, u \,  v \, dxdy, \qquad \epsilon = 0,
\end{gathered}
\end{equation}
where the family $V_{\epsilon}$ satisfies (\ref{hip1}) and (\ref{hip2}). Thus, we can define the  linear operators $A_\epsilon: H^1 \subset H^{-1}(\Omega) \mapsto H^{-1}(\Omega)$ by $\left\langle A_\epsilon u,v \right\rangle_{-1,1} = a_\epsilon(u,v)$, for all $v \in H^1(\Omega)$ and $\epsilon \in [0,\epsilon_0]$.

 The operators $A_\epsilon$ can also be considered as going from
 $H^{2-\alpha}(\Omega) \subset H^{-\alpha}(\Omega)$ into $H^{-\alpha}(\Omega)$,
 for $\frac{1}{2} < \alpha \leqslant 1$.  
 Abusing the  notation, we will sometimes denote all these different realizations
 simply by   $A_{\epsilon}$. 

\newtheorem{lam}{Lemma}[section]
\begin{lam}
\label{lam}
There exists $\lambda^* \in \mathbb{R}$, independent of $0 \leqslant \epsilon \leqslant \epsilon_0$, such that the bilinear form $a_\epsilon$ is uniformly coercive in   $H^{1}(\Omega)$ for all $\lambda > \lambda^*$.
 Consequently,  the operators $A_\epsilon$  are continuously invertible from $H^{2-\alpha}(\Omega)$ into $H^{-\alpha}(\Omega)$, for all $\epsilon \in[0, \epsilon_0]$ and $\frac{1}{2} < \alpha \leqslant 1$. 
\end{lam} 
\begin{proof}
Let us consider the case $a_\epsilon$ for $\epsilon > 0$. A similar argument gives the result for $a_{0}$.
First, we note that 
\begin{eqnarray}
\label{t1}
a_{\epsilon}(\phi,\phi) \geqslant \left\|\nabla \phi\right\|^{2}_{L^{2}(\Omega)}+\lambda \left\|\phi\right\|^2_{L^{2}(\Omega)} - \frac{1}{\epsilon} \int_{\omega_{\epsilon}}\left(V_{\epsilon}\right)_{-} \, \left|\phi\right|^2 \, d\xi, 
\end{eqnarray}
where $\left(V_{\epsilon}\right)_{-}$ is the negative part of the potential $V_{\epsilon}$ satisfying $V_{\epsilon} = \left(V_{\epsilon}\right)_{+} - \left(V_{\epsilon}\right)_{-}$. For the negative part 
\begin{eqnarray*}
 \frac{1}{\epsilon} \int_{\omega_{\epsilon}} \left(V_{\epsilon}\right)_{-} \, \left|\phi\right|^2 \, d\xi  \leqslant  \left(\frac{1}{\epsilon}\int_{\omega_{\epsilon}}\left|V_{\epsilon}\right|^2 \, d\xi \right)^{\frac{1}{2}} \left(\frac{1}{\epsilon}\int_{\omega_{\epsilon}}\left|\phi\right|^4 \, d\xi \right)^{\frac{1}{2}}  
 \leqslant  C \left(\frac{1}{\epsilon}\int_{\omega_{\epsilon}}\left|\phi\right|^4 d\xi \right)^{\frac{1}{2}}.
\end{eqnarray*}
Taking $\frac{1}{2}<s<1$ and $s-1\geqslant -\frac{1}{4}$, that is, $\frac{3}{4}\leqslant s<1$, and using the Lemma \ref{lem1} with $q=4$, we get 
$$
\frac{1}{\epsilon}\int_{\omega_{\epsilon}}\left(V_{\epsilon}\right)_{-}\left|\phi\right|^2 \, d\xi  \leqslant C\left\|\phi\right\|^{2}_{H^{s}(\Omega)}\leqslant C \left\|\phi\right\|^{2s}_{H^{1}(\Omega)} \left\|\phi\right\|^{2(1-s)}_{L^{2}(\Omega)}.
$$
Due to Young's Inequality, we obtain for any $\delta>0$ 
\begin{eqnarray}
\label{t2}
\frac{1}{\epsilon}\int_{\omega_{\epsilon}}\left(V_{\epsilon}\right)_{-}\left|\phi\right|^2 d\xi  \leqslant \delta \left\|\phi\right\|^{2}_{H^{1}(\Omega)} + C_{\delta}\left\|\phi\right\|^{2}_{L^{2}(\Omega)}.
\end{eqnarray}
Then, it follows from (\ref{t1}) and (\ref{t2}) that
$$
a_{\epsilon}(\phi,\phi) \geqslant  \left( \lambda-(1+C_{\delta})\right)\left\|\phi\right\|^2_{L^{2}(\Omega)}+(1-\delta)\left\|\phi\right\|^{2}_{H^{1}(\Omega)}.
$$ 
Consequently, we can take $\delta>0$ small enough and $\lambda>0$ large enough such that
$$
a_{\epsilon}(\phi,\phi)\geqslant C \left\|\phi\right\|^{2}_{H^{1}(\Omega)}, \qquad \mbox{$\forall$ $\phi\in H^{1}(\Omega)$}.
$$
Hence, the bilinear form $a_{\epsilon}$ is uniformly coercive, and we can take any $\lambda>0$ if $V_{\epsilon}\geqslant0$ in (\ref{CF}).
\end{proof}

\newtheorem{obsabs2}[lam]{Remark}
\begin{obsabs2}
\label{obsabs2}
For each $\epsilon \in [0,\epsilon_0]$,  the linear operator $A_{\epsilon}:H^{2-\alpha}(\Omega) \subset  H^{-\alpha}(\Omega)\mapsto H^{-\alpha}(\Omega)$
 ($ \frac{1}{2} < \alpha \leqslant 1$) is 
a selfadjoint, thus sectorial operator with spectrum contained in the subset $(\lambda,\infty) \subset \mathbb{R}$ for $\lambda > \lambda^* >0$. 
\end{obsabs2}

\newtheorem{obsabs1}[lam]{Remark}
\begin{obsabs1}
\label{obsabs1}
For each $\epsilon \in [0,\epsilon_0]$,  the linear operator $A^{-1}_{\epsilon}:H^{-\alpha}(\Omega) \mapsto H^{2-\alpha}(\Omega)$ is continuous and therefore, compact as an operator from $H^{-\alpha}(\Omega)$ into $H^{2-\beta}(\Omega)$,
 if $\beta > \alpha$. 
\end{obsabs1}

We now define $F_{\epsilon}: H^{1}(\Omega) \mapsto H^{-\alpha}(\Omega)$, with $\frac{1}{2}<\alpha< 1$, by
\begin{eqnarray}
\label{defFe}
\left\langle F_{0}(u),\phi \right\rangle:= \int_{\partial \Omega}\mu \, \gamma\left(f(u)\right)\gamma(\phi) \, dS \qquad \mbox{and} \qquad \left\langle F_{\epsilon}(u),\phi \right\rangle:= \frac{1}{\epsilon}\int_{\omega_{\epsilon}}f(u)\phi \, d\xi, \qquad \mbox{with $0<\epsilon \leqslant \epsilon_{0}$},
\end{eqnarray}
for $u\in H^{1}(\Omega)$ and $\phi\in H^{\alpha}(\Omega)$, where $\gamma$ denotes the trace operator and $\mu$ is the mean value of $g(s,\cdot)$ at $s \in (0,T)$ introduced in \eqref{mu}.

Using the hypothesis (H), we have by Lemma \ref{lem1eq} below that $F_{\epsilon}$ is well defined for each $0\leqslant\epsilon\leqslant\epsilon_{0}$. By results in 
\cite{OP}  the problems (\ref{1}) and (\ref{2}) are ``equivalent'' to the following  abstract form
\begin{eqnarray}
\label{probabstract}
\left\{ \begin{array}{ll}
                    \dot{u}_{\epsilon}(t)+A_{\epsilon}u_\epsilon(t)=F_{\epsilon}\left(u_\epsilon(t)\right), & \qquad t>0\quad \mbox{and} \quad 0\leqslant \epsilon\leqslant \epsilon_{0} \\
                   u_\epsilon(0)=\phi^{\epsilon}.
        \end{array} \right. 
\end{eqnarray}
As previously mentioned, it is known that the parabolic problems \eqref{probabstract} are well posed in $H^{1}(\Omega)$   and, for each $0\leqslant \epsilon \leqslant \epsilon_{0}$, they determine a nonlinear semigroup 
$
T^{\epsilon}(t)\phi=u_\epsilon(t,\phi),
$
for $t>0$ and $\phi \in H^{1}(\Omega)$, associated to the equations (\ref{1}) and (\ref{2}).
Moreover, under assumption (\ref{dissipative}), the problems \eqref{probabstract} have a global attractor $\mathscr{A}_{\epsilon}$ uniformly bounded in $L^{\infty}(\Omega)$ (see \cite{alexandreat,OP}).

We now  obtain some estimates for the family of operators $\{ A_\epsilon \}_{\epsilon \in [0,\epsilon_0]}$.

\newtheorem{lem4}[lam]{Lemma}
\begin{lem4}
\label{lem4}
If $ \frac{1}{2} < \alpha \leqslant 1$ then there exists a $K(\epsilon) \geqslant 0$, $K(\epsilon) \to 0$ as $\epsilon \to 0$, such that 
$$
\| \left( A_\epsilon - A_0 \right) \, u \|_{H^{-\alpha}(\Omega)} \leqslant K(\epsilon) \, \| A_0 u \|_{H^{-\alpha}(\Omega)}, \qquad \textrm{for all } u \in H^{2-\alpha}(\Omega).
$$
\end{lem4}
\begin{proof}
This estimate is a direct consequence of Lemmas \ref{lem3} and \ref{lam}. Indeed, by Lemma \ref{lem3}, we have 
\begin{eqnarray*}
\left\langle \left( A_\epsilon - A_0 \right) u,\varphi \right\rangle_{-\alpha,\alpha} & = & \frac{1}{\epsilon}\int_{\omega_{\epsilon}}V_{\epsilon} \, u \, \varphi \, d\xi - 
\int_{\partial \Omega} V_{0} \, u \, \varphi \, dS \\
& = & \left\langle \left( P_{\epsilon} - P_{0} \right) (u),\varphi\right\rangle_{-\alpha,\alpha} 
 \leqslant  \| P_{\epsilon} - P_{0} \|_{\mathscr{L}(H^{2-\alpha},H^{-\alpha})} \, \|u\|_{H^{2-\alpha}(\Omega)} \, \|\varphi\|_{H^{\alpha}(\Omega)},
\end{eqnarray*}
with $\| P_{\epsilon} - P_{0} \|_{\mathscr{L}(H^{2-\alpha},H^{-\alpha})} \to 0$ as $\epsilon \to 0$. Since $\|u\|_{H^{2-\alpha}(\Omega)} \leqslant C \|A_0u\|_{H^{-\alpha}(\Omega)} $ by Lemma \ref{lam}, the result follows.
\end{proof}

Now we get a convergence result for the linear semigroup $e^{-tA_\epsilon}$ as $\epsilon$ goes to zero.

\newtheorem{linear}[lam]{Proposition}
\begin{linear}
\label{linear}
If $ \frac{1}{2} < \alpha \leqslant 1$ then the family of linear semigroups $e^{ -tA_{\epsilon}}$ satisfies
\begin{gather*}
\| A^{\beta}_{0} \left( e^{-tA_{\epsilon}}-  e^{-tA_0} \right) \|  \leqslant
C(\epsilon)
\frac{1}{ t^{\beta}} e^{-bt}, \qquad  0 \leqslant \beta \leqslant 1
\end{gather*}
for $t >0$,  where  $b \in \mathbb{R}$ can be chosen as close to  $\lambda > \lambda^*$ as needed  and $C(\epsilon) \to 0$  as $\epsilon \to 0$.
\end{linear}
\begin{proof}
Since $D(A_\epsilon) = D(A_0) = H^{2-\alpha}(\Omega)$ and the family of operators $\{ A_\epsilon \}_{\epsilon \in [0,\epsilon_0]}$ satisfies Lemma \ref{lem4}, we obtain the result as a directly consequence of \cite[Theorem 3.3]{antoniomarcone}.
\end{proof}

Next we study the behavior of the maps $F_{\epsilon}$ defined in (\ref{defFe}).

\newtheorem{lem1eq}[lam]{Lemma}
\begin{lem1eq}
\label{lem1eq}
Suppose that (H) holds and $ \frac{1}{2} < \alpha < 1$. Then:
\begin{enumerate}

\item There exists $k>0$ independent of $\epsilon$ such that
$$
\left\|F_{\epsilon}(u)\right\|_{H^{-\alpha}(\Omega)}\leqslant k, \qquad \mbox{$\forall$ $u\in H^{1}(\Omega)$} \quad \mbox{and}\quad 0\leqslant\epsilon\leqslant\epsilon_{0}. 
$$

\item For each $0\leqslant\epsilon\leqslant \epsilon_{0}$, the map $F_{\epsilon}:H^{1}(\Omega)\mapsto H^{-\alpha}(\Omega)$ is globally Lipschitz, uniformly in $\epsilon$.

\item For each $u\in H^{1}(\Omega)$, we have
$$
\left\|F_{\epsilon}(u)-F_{0}(u)\right\|_{H^{-\alpha}(\Omega)}\to 0, \qquad \mbox{as $\epsilon\to0$}.
$$
Furthermore, this limit is uniform for $u$ in bounded set of $H^{1}(\Omega)$.

\item If $u_{\epsilon}\to u$ in $H^{1}(\Omega)$, as $\epsilon\to 0$, then
$$
\left\|F_{\epsilon}(u_{\epsilon})-F_{0}(u)\right\|_{H^{-\alpha}(\Omega)}\to 0, \qquad \mbox{as $\epsilon\to0$}.
$$
\end{enumerate}
\end{lem1eq}
\begin{proof}
$1.$ For each $u\in H^{1}(\Omega)$ and $0\leqslant\epsilon\leqslant\epsilon_{0}$, we have
$$
\left\|F_{\epsilon}(u)\right\|_{H^{-\alpha}(\Omega)}=\displaystyle \sup_{\left\|\phi\right\|_{H^{\alpha}(\Omega)}=1} 
\left|\left\langle F_{\epsilon}(u),\phi \right\rangle\right|.
$$

Using (H) and the Lemma~\ref{lem1}, we have that for each $0<\epsilon\leqslant\epsilon_{0}$ and $\phi\in H^{\alpha}(\Omega)$, 
$$
\left|\left\langle F_{\epsilon}(u),\phi \right\rangle\right| \leqslant  \left(\frac{1}{\epsilon}\int_{\omega_{\epsilon}}\left|f(u(x))\right|^2 dx\right)^{\frac{1}{2}} \left(\frac{1}{\epsilon}\int_{\omega_{\epsilon}}\left|\phi(x)\right|^{2}dx\right)^{\frac{1}{2}} \leqslant CK \left\|\phi\right\|_{H^{\alpha}(\Omega)}.
$$
We note that $C$ does not depend of $\epsilon$, because the set $\omega_{\epsilon}$ has Lebesgue measure $\left|\omega_{\epsilon}\right|\leqslant \left|r_{\epsilon}\right|$, where $r_{\epsilon}$ is given by (\ref{repsilon}) e $\left|r_{\epsilon}\right|=O(\epsilon)$. Hence, there exists a constant $k>0$ independent of $\epsilon$ such that
$$
\left\|F_{\epsilon}(u)\right\|_{H^{-\alpha}(\Omega)}\leqslant k, \qquad \mbox{$\forall$ } 0<\epsilon\leqslant\epsilon_{0}. 
$$ 

Now, using (H) and the continuity of the trace operator $\gamma: H^{\alpha}(\Omega) \mapsto L^{2}(\partial \Omega)$, we get
$$
\left|\left\langle F_{0}(u),\phi \right\rangle\right| \leqslant  \|\mu\|_{L^\infty(\partial \Omega)}\left(\int_{\partial \Omega}\left|\gamma(f(u(x)))\right|^2 dx\right)^{\frac{1}{2}} \left(\int_{\partial \Omega}\left|\gamma(\phi(x))\right|^{2}dx\right)^{\frac{1}{2}} \leqslant K \left\|\gamma(\phi)\right\|_{L^{2}(\partial \Omega)} \leqslant c \, K \left\|\phi\right\|_{H^{\alpha}(\Omega)}.
$$
Hence, there exists $k>0$ such that
$$
\left\|F_{0}(u)\right\|_{H^{-\alpha}(\Omega)}\leqslant k. 
$$ 

\noindent $2.$ Let $u,v\in H^{1}(\Omega)$ and $0\leqslant\epsilon\leqslant\epsilon_{0}$, we have
$$
\left\|F_{\epsilon}(u)-F_{\epsilon}(v)\right\|_{H^{-\alpha}(\Omega)}=\displaystyle \sup_{\left\|\phi\right\|_{H^{\alpha}(\Omega)}=1} 
\left|\left\langle F_{\epsilon}(u)-F_{\epsilon}(v),\phi \right\rangle\right|.
$$

For each $0<\epsilon\leqslant\epsilon_{0}$ and $\phi\in H^{\alpha}(\Omega)$, from Lemma~\ref{lem1} we have
$$
\begin{array}{rcl}
\displaystyle \left|\left\langle F_{\epsilon}(u)-F_{\epsilon}(v),\phi \right\rangle\right| &\leqslant & \displaystyle \left(\frac{1}{\epsilon}\int_{\omega_{\epsilon}}\left|f(u(x))-f(v(x))\right|^{2}dx\right)^{\frac{1}{2}} \left(\frac{1}{\epsilon}\int_{\omega_{\epsilon}}\left|\phi(x)\right|^{2}dx\right)^{\frac{1}{2}}\\
&\leqslant & \displaystyle C \left(\frac{1}{\epsilon}\int_{\omega_{\epsilon}} \left|f(u(x))-f(v(x))\right|^{2} dx\right)^{\frac{1}{2}}\left\|\phi\right\|_{H^{\alpha}(\Omega)}.
\end{array}
$$
Using (H) and the Lemma~\ref{lem1}, we have 
$$
\begin{array}{rcl}
\displaystyle \left\|F_{\epsilon}(u)-F_{\epsilon}(v)\right\|_{H^{-\alpha}(\Omega)} & \leqslant & \displaystyle C \left(\frac{1}{\epsilon}\int_{\omega_{\epsilon}} \left|f'(\theta(x)u(x)+(1-\theta(x))v(x))\right|^{2} \left|u(x)-v(x)\right|^{2} dx\right)^{\frac{1}{2}} \\
& \leqslant & \displaystyle CK \left(\frac{1}{\epsilon}\int_{\omega_{\epsilon}}\left|u(x)-v(x)\right|^{2} dx\right)^{\frac{1}{2}}\leqslant \displaystyle CK \left\|u-v\right\|_{H^{1}(\Omega)},
\end{array}
$$
for some $0\leqslant \theta(x)\leqslant 1$, $x\in \bar{\Omega}$. Hence, there exists $L>0$ independent of $\epsilon$ such that
$$
\left\|F_{\epsilon}(u)-F_{\epsilon}(v)\right\|_{H^{-\alpha}(\Omega)}\leqslant L \left\|u-v\right\|_{H^{1}(\Omega)}.
$$
Therefore, for each $0<\epsilon\leqslant\epsilon_{0}$, $F_{\epsilon}$ is globally Lipschitz, uniformly in $\epsilon$. Similarly, $F_{0}$ is globally Lipschitz.

\vspace{0.4cm}

\noindent $3.$ Initially, we take $\alpha_{0}$ satisfying $\frac{1}{2}< \alpha_{0} < 1$. For each $u\in H^{1}(\Omega)$ and $\phi\in H^{\alpha_{0}}(\Omega)$, we have
$$
\left| \left\langle F_{\epsilon}(u),\phi\right\rangle-\left\langle F_{0}(u),\phi\right\rangle \right|=\left|\frac{1}{\epsilon}\int_{\omega_{\epsilon}}f(u(x))\phi(x)dx-\int_{\partial \Omega} \mu \,  \gamma\left(f(u(x))\right)\gamma\left(\phi(x)\right)dx\right|.
$$
From Lemma~\ref{lem2}, we get that for each $\phi\in H^{\alpha_{0}}(\Omega)$,
\begin{eqnarray}
\label{convunifor1}
\left\langle F_{\epsilon}(u),\phi\right\rangle \to \left\langle F_{0}(u),\phi\right\rangle, \qquad \mbox{as $\epsilon\to0$.}
\end{eqnarray}

Moreover, fixing $u\in H^{1}(\Omega)$ and using the item 1, we have that the set $\{ F_{\epsilon}(u)\in H^{-\alpha_{0}}(\Omega) \; : \; \epsilon\in(0, \epsilon_{0}] \}$ is equicontinuous. Thus, the limit (\ref{convunifor1}) is uniform for $\phi$ in compact sets of $H^{\alpha_{0}}(\Omega)$. Hence, choosing $\alpha_{0}$ such that $\frac{1}{2}<\alpha_{0}<\alpha <1$, we have that the embedding $H^{\alpha}(\Omega)\hookrightarrow H^{\alpha_{0}}(\Omega)$ is compact, and then, in particular,  
\begin{eqnarray}
\label{convpont1}
\left\|F_{\epsilon}(u)-F_{0}(u)\right\|_{H^{-\alpha}(\Omega)}=
\sup_{\left\|\phi\right\|_{H^{\alpha}(\Omega)}=1}
\left| \left\langle F_{\epsilon}(u)-F_{0}(u),\phi\right\rangle \right| \to 0,\qquad \mbox{as $\epsilon\to0$}. 
\end{eqnarray}

Now, we will show that the limit (\ref{convpont1}) is uniform for $u\in H^{1}(\Omega)$, $\left\|u\right\|_{H^{1}(\Omega)}\leqslant R$ for some $R>0$. Initially, we show that $F_{\epsilon}$ is continuous in $H^{1}(\Omega)$ space with the weak topology.
Let $u_{n}\rightharpoonup u_{0}$ in $H^{1}(\Omega)$, as $n\to \infty$. Since $H^{1}(\Omega)\hookrightarrow H^{s}(\Omega)$ with compact embedding, for $s<1$, we have
$$
u_{n} \to u_{0}\quad \mbox{in}\quad H^{s}(\Omega), \quad \mbox{as $n\to\infty$}.
$$
For each $\phi\in H^{\alpha}(\Omega)$, it follows from Lemma~\ref{lem1} that
$$
\begin{array}{rcl}
\displaystyle \left|\left\langle F_{\epsilon}(u_{n})-F_{\epsilon}(u_{0}),\phi \right\rangle\right| &\leqslant & \displaystyle \left(\frac{1}{\epsilon}\int_{\omega_{\epsilon}}\left|f(u_{n}(x))-f(u_{0}(x))\right|^{2}dx\right)^{\frac{1}{2}} \left(\frac{1}{\epsilon}\int_{\omega_{\epsilon}}\left|\phi(x)\right|^{2}dx\right)^{\frac{1}{2}}\\
&\leqslant & \displaystyle C \left(\frac{1}{\epsilon}\int_{\omega_{\epsilon}} \left|f(u_{n}(x))-f(u_{0}(x))\right|^{2} dx\right)^{\frac{1}{2}}\left\|\phi\right\|_{H^{\alpha}(\Omega)}.
\end{array}
$$
Using (H) and the Lemma~\ref{lem1}  with $\frac{1}{2}<s<1$, we have for some $0\leqslant \theta(x)\leqslant 1$, $x\in \bar{\Omega}$, that
$$
\begin{array}{rcl}
\displaystyle \left\|F_{\epsilon}(u_{n})-F_{\epsilon}(u_{0})\right\|_{H^{-\alpha}(\Omega)} & \leqslant & \displaystyle C \left(\frac{1}{\epsilon}\int_{\omega_{\epsilon}} \left|f'(\theta(x)u_{n}(x)+(1-\theta(x))u_{0}(x))\right|^{2} \left|u_{n}(x)-u_{0}(x)\right|^{2} dx\right)^{\frac{1}{2}} \\
& \leqslant & \displaystyle CK \left(\frac{1}{\epsilon}\int_{\omega_{\epsilon}}\left|u_{n}(x)-u_{0}(x)\right|^{2} dx\right)^{\frac{1}{2}}\leqslant \displaystyle CK \left\|u_{n}-u_{0}\right\|_{H^{s}(\Omega)}\to0,\qquad \mbox{as $n\to\infty$}.
\end{array}
$$

Therefore, for each $0<\epsilon\leqslant\epsilon_{0}$, $F_{\epsilon}:H^{1}(\Omega)\mapsto H^{-\alpha}(\Omega)$ is continuous in $H^{1}(\Omega)$ with the weak topology. Hence, $F_{\epsilon}$ is uniformly continuous in compact sets of $H^{1}(\Omega)$ with the weak topology. We note that the closed ball $\bar{B}_{R}(0)=\{u\in H^{1}(\Omega) \; : \; \left\|u\right\|_{H^{1}(\Omega)}\leqslant R\}$, with $R>0$, is compact in $H^{1}(\Omega)$ with the weak topology. From this and (\ref{convpont1}), we get that the limit (\ref{convpont1}) is uniform in $\bar{B}_{R}(0)$.

\vspace{0.4cm}

\noindent $4.$  This item follows from 2. and 3. adding and subtracting $F_{\epsilon}(u)$.
\end{proof}

To obtain the lower semicontinuity of attractors, we also need to analyze the linearized problems. Hence, it is necessary to study the properties of the differential of $F_{\epsilon}$. 

\newtheorem{not1eq}[lam]{Lemma}
\begin{not1eq}
\label{not1eq}   
Suppose that (H) holds and $ \frac{1}{2} < \alpha < 1$. Then, for each $0\leqslant\epsilon\leqslant \epsilon_{0}$, $F_{\epsilon}:H^{1}(\Omega) \mapsto H^{-\alpha}(\Omega)$ is Fr\'echet differentiable, uniformly in $\epsilon$, with Fr\'echet differential given by 
$$
\begin{array}{lll}
F'_{\epsilon}:H^{1}(\Omega) \mapsto \mathscr{L}\left(H^{1}(\Omega),H^{-\alpha}(\Omega)\right) \\
\qquad \quad \ \ u^{*} \rightarrow F'_{\epsilon}(u^{*}): H^{1}(\Omega) \mapsto  H^{-\alpha}(\Omega)\\
\qquad \qquad \qquad \qquad \qquad \qquad  w  \rightarrow F'_{\epsilon}(u^{*})w 
\end{array}
$$
where $\mathscr{L}\left(H^{1}(\Omega),H^{-\alpha}(\Omega)\right)$ denotes the space of the continuous linear operators from $H^{1}(\Omega)$ in $H^{-\alpha}(\Omega)$, and
$$
\label{derivadas}
\begin{gathered}
\left\langle F'_{\epsilon}(u^{*})w,\phi\right\rangle := \displaystyle\frac{1}{\epsilon}\int_{\omega_{\epsilon}}f'\left(u^{*}\right)w\phi \, d\xi, \qquad \mbox{$\forall$ $\phi\in H^{\alpha}(\Omega)$}\quad \mbox{and}\quad 0<\epsilon\leqslant \epsilon_{0}, \\
\left\langle F'_{0}(u^{*})w,\phi\right\rangle := \displaystyle\int_{\partial \Omega}\mu \, \gamma\left(f'\left(u^{*}\right)w\right)\gamma\left(\phi\right) \, dS, \qquad \mbox{$\forall$ $\phi\in H^{\alpha}(\Omega)$,}
\end{gathered}
$$
where $\gamma$ denotes the trace operator and $\mu$ is the mean value given by \eqref{mu}. 
\end{not1eq}
\begin{proof}
From (H), in particular, we have that $f\in C^{2}(\mathbb{R})$, hence $f'(v)\in \mathscr{L}(\mathbb{R})$, for each $v\in \mathbb{R}$. Using this and the linearity of integral and of trace operator, we get that for each $0\leqslant \epsilon\leqslant \epsilon_{0}$, $F'_{\epsilon}(u^{*})\in \mathscr{L}\left(H^{1}(\Omega),H^{-\alpha}(\Omega)\right)$, for each $u^{*}\in H^{1}(\Omega)$. 

Now, we will show that given $\eta>0$, there exists $\delta>0$ independent of $\epsilon$ such that 
$$
\left\|F_{\epsilon}(u^{*}+w)-F_{\epsilon}(u^{*})-F'_{\epsilon}(u^{*})w\right\|_{H^{-\alpha}(\Omega)}\leqslant \eta \left\|w\right\|_{H^{1}(\Omega)}, \qquad \mbox{$\forall$ $w\in H^{1}(\Omega)$}\quad \mbox{with}\quad \left\|w\right\|_{H^{1}(\Omega)}\leqslant \delta.
$$

In fact, for each $0<\epsilon\leqslant\epsilon_{0}$, $w\in H^{1}(\Omega)$ and $\phi\in H^{\alpha}(\Omega)$, from Lemma~\ref{lem1} we have
\begin{eqnarray*}
& &\left|\left\langle F_{\epsilon}(u^{*}+w)-F_{\epsilon}(u^{*})-F'_{\epsilon}(u^{*})w,\phi \right\rangle\right| \\
\qquad &\leqslant & \displaystyle \left(\frac{1}{\epsilon}\int_{\omega_{\epsilon}}\left|f(u^{*}(x)+w(x))-f(u^{*}(x))-f'(u^{*}(x))w(x)\right|^{2}dx\right)^{\frac{1}{2}} \left(\frac{1}{\epsilon}\int_{\omega_{\epsilon}}\left|\phi(x)\right|^{2}dx\right)^{\frac{1}{2}}\\
\qquad &\leqslant & \displaystyle C \left(\frac{1}{\epsilon}\int_{\omega_{\epsilon}} \left|f(u^{*}(x)+w(x))-f(u^{*}(x))-f'(u^{*}(x))w(x)\right|^{2} dx\right)^{\frac{1}{2}}\left\|\phi\right\|_{H^{\alpha}(\Omega)}.
\end{eqnarray*}
Using (H) and the Lemma~\ref{lem1}, we have
\begin{eqnarray*}
& & \left\|F_{\epsilon}(u^{*}+w)-F_{\epsilon}(u^{*})-F'_{\epsilon}(u^{*})w\right\|_{H^{-\alpha}(\Omega)} \\
& \leqslant & \displaystyle C \left(\frac{1}{\epsilon}\int_{\omega_{\epsilon}} \left|f'(u^{*}(x)+\theta(x) w(x))-f'(u^{*}(x))\right|^{2} \left|w(x)\right|^{2} dx\right)^{\frac{1}{2}} \\
& \leqslant & \displaystyle C \left[\left(\frac{1}{\epsilon}\int_{\omega_{\epsilon}} \left|f'(u^{*}(x)+\theta(x) w(x))-f'(u^{*}(x))\right|^{4} dx\right)^{\frac{1}{2}}\left(\frac{1}{\epsilon}\int_{\omega_{\epsilon}}\left|w(x)\right|^{4}dx \right)^{\frac{1}{2}}\right]^{\frac{1}{2}}\\
& \leqslant & \displaystyle C \left[\left(\frac{1}{\epsilon}\int_{\omega_{\epsilon}} \left|f''\left[ s(x)(u^{*}(x)+\theta(x) w(x))+(1-s(x))u^{*}(x) \right]\right|^{4} \left|\theta(x)w(x)\right|^{4} dx\right)^{\frac{1}{2}}\left\|w\right\|^{2}_{H^{1}(\Omega)}\right]^{\frac{1}{2}}\\
& \leqslant & \displaystyle C K \left(\frac{1}{\epsilon}\int_{\omega_{\epsilon}} \left|w(x)\right|^{4} dx\right)^{\frac{1}{4}}\left\|w\right\|_{H^{1}(\Omega)} \leqslant \displaystyle C K \left\|w\right\|_{H^{1}(\Omega)} \left\|w\right\|_{H^{1}(\Omega)}, 
\end{eqnarray*}
for some $0\leqslant \theta(x)\leqslant 1$ and $0\leqslant s(x)\leqslant 1$, $x\in \bar{\Omega}$.

Therefore, given $\eta>0$, taking $\delta=\frac{\eta}{CK}>0$ we get that for $\left\|w\right\|_{H^{1}(\Omega)}\leqslant \delta$, 
$$
\left\|F_{\epsilon}(u^{*}+w)-F_{\epsilon}(u^{*})-F'_{\epsilon}(u^{*})w\right\|_{H^{-\alpha}(\Omega)}\leqslant \eta \left\|w\right\|_{H^{1}(\Omega)}.
$$
We note that $\delta$ does not depend of $\epsilon$. Hence, for each $0<\epsilon\leqslant \epsilon_{0}$, $F_{\epsilon}$ is Fr\'echet differentiable, uniformly in $\epsilon$. Similarly, $F_{0}$ is also Fr\'echet differentiable. 
\end{proof}

Similarly,  we can prove the following lemma:

\newtheorem{lem2eq}[lam]{Lemma}
\begin{lem2eq}
\label{lem2eq}
Suppose that (H) holds and $ \frac{1}{2} < \alpha < 1$. Then:
\begin{enumerate}
\item There exists $k>0$ independent of $\epsilon$ such that
$$
\left\|F'_{\epsilon}(u^{*})\right\|_{\mathscr{L}(H^{1},H^{-\alpha})}\leqslant k, \qquad \mbox{$\forall$ $u^{*}\in H^{1}(\Omega)$} \quad \mbox{and}\quad 0\leqslant\epsilon\leqslant\epsilon_{0}. 
$$

\item For each $0\leqslant\epsilon\leqslant \epsilon_{0}$, the map $F'_{\epsilon}:H^{1}(\Omega) \mapsto \mathscr{L}(H^{1}(\Omega),H^{-\alpha}(\Omega))$ is globally Lipschitz, uniformly in $\epsilon$.

\item For each $u^{*}\in H^{1}(\Omega)$, we have
$$
\left\|F'_{\epsilon}(u^{*})-F'_{0}(u^{*})\right\|_{\mathscr{L}(H^{1},H^{-\alpha})}\to 0, \qquad \mbox{as $\epsilon\to0$}.
$$

\item If $u^{*}_{\epsilon}\to u^{*}$ in $H^{1}(\Omega)$, as $\epsilon\to 0$, then
$$
\left\|F'_{\epsilon}(u^{*}_{\epsilon})-F'_{0}(u^{*})\right\|_{\mathscr{L}(H^{1},H^{-\alpha})}\to 0, \qquad \mbox{as $\epsilon\to0$}.
$$

\item If $u^{*}_{\epsilon}\to u^{*}$ in $H^{1}(\Omega)$, as $\epsilon\to 0$, and $w_{\epsilon}\to w$ in $H^{1}(\Omega)$, as $\epsilon\to 0$, then
$$
\left\|F'_{\epsilon}(u^{*}_{\epsilon})w_{\epsilon}-F'_{0}(u^{*})w\right\|_{H^{-\alpha}(\Omega)}\to 0, \qquad \mbox{as $\epsilon\to0$}.
$$
\end{enumerate}
\end{lem2eq}

\section{Upper semicontinuity of attractors and continuity of equilibria} 
\label{continuityeq}

\ \ \ \ The \emph{upper} semicontinuity of the family of attractors 
 $\{\mathscr{A}_{\epsilon}\}_{\epsilon\in[0,\epsilon_{0}]}$  of  (\ref{1}) and (\ref{2}) is easily obtained using the results of \cite{antoniomarcone}. 

\newtheorem{upperattrac}{Proposition}[section]
\begin{upperattrac}
\label{upperattrac}
Suppose that (H) holds.  Then there exists $\epsilon_{0}>0$ such that:
\begin{enumerate}
\item The problems (\ref{1}) and (\ref{2}) have a global attractor $\mathscr{A}_{\epsilon}$ in $H^{1}(\Omega)$ for each $0\leqslant\epsilon\leqslant\epsilon_{0}$. Moreover, there exists $R>0$ independent of $\epsilon$ such that
$$ 
\sup_{\epsilon\in[0,\epsilon_{0}]}\sup_{u\in \mathscr{A}_{\epsilon}}\left\|u\right\|_{H^{1}(\Omega)\cap L^{\infty}(\Omega)}\leqslant R.
$$
In particular, $\mathscr{A}_{0}$ attracts $\bigcup_{\epsilon\in (0,\epsilon_{0}]}\mathscr{A}_{\epsilon}$ in $H^{1}(\Omega)$.

\item Let $0<\tau<\infty$ and $B\subset H^{1}(\Omega)$ be a bounded set. For each $0\leqslant\epsilon\leqslant\epsilon_{0}$, let $\phi^{\epsilon}\in H^{1}(\Omega)$ such that $\phi^{\epsilon}\to\phi^{0}$ in $H^{1}(\Omega)$, as $\epsilon\to0$, with $\phi^{0}\in B$. Then, there exist $M(\tau)>0$ and a function $C(\epsilon)\geqslant0$, with $C(\epsilon)\to 0$ as $\epsilon\to0$, such that
$$
\left\|T^{\epsilon}(t)\phi^{\epsilon}-T^{0}(t)\phi^{0}\right\|_{H^{1}(\Omega)}\leqslant M(\tau)C(\epsilon)t^{-\gamma}, \qquad \mbox{for $t\in(0,\tau]$,}
$$
for some $\gamma\in (0,1)$.
\item The family of global attractors of (\ref{1}) and (\ref{2}), $\{\mathscr{A}_{\epsilon}\}_{\epsilon\in[0,\epsilon_{0}]}$, is upper semicontinuous at $\epsilon=0$
 in $H^{1}(\Omega)$:
$$
\sup_{u_{\epsilon}\in \mathscr{A}_{\epsilon}} \inf_{u_{0}\in \mathscr{A}_{0}} \{\left\|u_{\epsilon}-u_{0}\right\|_{H^{1}(\Omega)}\}\to 0, \qquad \mbox{as $\epsilon \to0$}.
$$
\end{enumerate}
\end{upperattrac}
\begin{proof}
It follows from Lemma \ref{lem4}, Proposition \ref{linear}
 and  \cite[Theorem 3.9]{antoniomarcone} (in the last reference, thought not
 explicitly stated,  the upper semicontinuity was proved in the phase space  $X^{\alpha}$).
\end{proof}

 For the \emph{lower} semicontinuity of the attractors, we  need to consider the set of   equilibria of the parabolic problem (\ref{probabstract}), which is the abstract version of (\ref{1}) and (\ref{2}). The equilibrium solutions of (\ref{1}) and (\ref{2}) are  the solutions of the respective abstract elliptic problems 
\begin{eqnarray}
\label{3e}
& & A_{\epsilon}u_{\epsilon}=F_{\epsilon}(u_{\epsilon}), \qquad 0<\epsilon\leqslant \epsilon_{0} \\
\label{4e}
& & A_{0}u_{0}=F_{0}(u_{0}), \qquad \epsilon=0.
\end{eqnarray}

Define $G_{\epsilon}:  H^{2-\alpha}(\Omega) \mapsto H^{-\alpha}(\Omega)$, with $\frac{1}{2}<\alpha< 1$, by
\begin{equation}
\label{defGe}
 G_{\epsilon}(u)    : =  A_{\epsilon} u -  F_{\epsilon}(u).  
\end{equation}
It follows from Lemma~\ref{not1eq} that $G_{\epsilon}$ is Fr\'echet differentiable.

The set of solutions of (\ref{3e}) and (\ref{4e}) is then given by 
$$
\mathscr{E}_{\epsilon}= \left\{ u\in H^{2-\alpha}(\Omega)   \; : \; 
G_{\epsilon}(u) = A_{\epsilon}u-F_{\epsilon}(u)=0 \right\}, \qquad \epsilon \in [0,\epsilon_0].
$$

Due to the gradient structure of the flow generated by (\ref{probabstract}), its  attractor is the unstable manifold of the  set $\mathscr{E}_{\epsilon}$   (see \cite{haleat}, for details). In particular, we must have $\mathscr{E}_{\epsilon} \neq \emptyset$.  Also, it follows from the regularization properties of the elliptic operator $A_{\epsilon}$ that $\mathscr{E}_{\epsilon}$ is a compact subset of  $ H^{1}(\Omega)$.

The \emph{upper} semicontinuity of the family of equilibria $\left\{\mathscr{E}_{\epsilon}\right\}_{\epsilon\in[0,\epsilon_{0}]}$ at $\epsilon=0$ in $H^1(\Omega)$ is a direct consequence of  Proposition~\ref{upperattrac}.
Indeed, if $\{ u_\epsilon \}_{\epsilon\in[0,\epsilon_{0}]}$  is a family of equilibria, we can extract a convergent subsequence in $H^1(\Omega)$ by item 3 of Proposition \ref{upperattrac}. Thus, using item 2 of Proposition \ref{upperattrac}, we can conclude that the limit function belongs to the set of equilibria $\mathscr{E}_{0}$. Hence we have:

\newtheorem{upperequil}[upperattrac]{Theorem}
\begin{upperequil}
\label{upperequil}
Suppose that (H) holds. Then, the family of equilibria $\{\mathscr{E}_{\epsilon}\}_{\epsilon\in [0,\epsilon_{0}]}$ is upper semicontinuous in $H^1(\Omega)$, at $\epsilon=0$.
\end{upperequil}

To get the lower semicontinuity of the family of equilibria, we assume an additional assumption.

\newtheorem{def3eq}[upperattrac]{Definition}
\begin{def3eq}
\label{def3eq}
We say that the solution $u^{*}_{\epsilon}$ of (\ref{3e}) and (\ref{4e}) is hyperbolic if zero does not belong to the spectrum set of the operator $A_{\epsilon}-F'_{\epsilon}(u^{*}_{\epsilon})$, that is, if $0\notin\sigma(A_{\epsilon}-F'_{\epsilon}(u^{*}_{\epsilon}))$. 
\end{def3eq}

\newtheorem{teo2eq}[upperattrac]{Proposition}
\begin{teo2eq}
\label{teo2eq}
Suppose that (H) holds. If all points in $\mathscr{E}_{0}$ are isolated, then there is only a finite number of them. Moreover, if $u^{*}_{0}$ is a hyperbolic solution of (\ref{4e}), then $u^{*}_{0}$ is isolated.
\end{teo2eq}
\begin{proof}
Since $\mathscr{E}_{0}$ is compact, we only need to prove that hyperbolic equilibria are isolated. Now observe that   $u^{*}_{0}$ is a hyperbolic solution of (\ref{4e}) if and only if it is a regular point of  the function 
 $G_{\epsilon}$, defined by $(\ref{defGe})$. Since   $G_{\epsilon}$ is Fr\'echet differentiable, the result follows 
 from the Inverse  Function Theorem.
\end{proof}
\newtheorem{teo4eq}[upperattrac]{Theorem}
\begin{teo4eq}
\label{teo4eq}
Suppose that (H) holds and that $u^{*}_{0}$ is a hyperbolic solution of
 (\ref{4e}).   Then, there exist $\epsilon_{0}>0$ and $\delta>0$ such that, for each $0<\epsilon\leqslant \epsilon_{0}$, the equation (\ref{3e}) has exactly one solution, $u^{*}_{\epsilon}$, in 
$\{v\in H^{1}(\Omega) \; : \;  \left\|v-u^{*}_{0}\right\|_{H^{1}(\Omega)}\leqslant \delta\}$.  Furthermore, 
$$
u^{*}_{\epsilon}\to u^{*}_{0} \quad \mbox{in $H^{1}(\Omega)$}, \qquad \mbox{as $\epsilon\to0$}.
$$
In particular, the family of equilibria $\left\{\mathscr{E}_{\epsilon}\right\}_{\epsilon\in[0,\epsilon_{0}]}$ is lower semicontinuous at $\epsilon=0$ in $H^{1}(\Omega)$.
\end{teo4eq}
\begin{proof}
Consider the function
  $G:[-\epsilon_0, \epsilon_0] \times  H^{2-\alpha}(\Omega) \mapsto H^{-\alpha}(\Omega)$, with $\frac{1}{2}<\alpha< 1$, by
$$
\begin{array}{rr}
G(\epsilon,u) := G_{\epsilon}(u) \quad \mbox{\rm if } \epsilon >0,  \\
 := G_{0}(u) \quad \mbox{\rm if } \epsilon \leqslant 0.  
\end{array}
$$
Since $G$ is Fr\'echet differentiable in $u$ and continuous in $\epsilon$, the result follows from the  Implicit Function Theorem (see~\cite[Theorem 9.3-Chapter 4]{loomis}).
\end{proof}

\newtheorem{teo5eq}[upperattrac]{Theorem}
\begin{teo5eq}
\label{teo5eq}
Suppose that (H) holds. If all solutions $u^{*}_{0}$ of (\ref{4e})
 are hyperbolic, then (\ref{4e}) has a finite number $m$ of solutions, $u^{*}_{0,1},...,u^{*}_{0,m}$, and there exists $\epsilon_{0}>0$ such that, for each $0<\epsilon\leqslant \epsilon_{0}$, the equation (\ref{3e}) has exactly $m$ solutions, $u^{*}_{\epsilon,1},...,u^{*}_{\epsilon,m}$. Moreover, for all $i=1,...,m$,
$$
u^{*}_{\epsilon,i}\to u^{*}_{0,i}\quad \mbox{in}\quad H^{1}(\Omega), \qquad \mbox{as $\epsilon\to0$}.
$$
\end{teo5eq}
\begin{proof} 
The proof follows from Proposition~\ref{teo2eq} and Theorem~\ref{teo4eq}. 
\end{proof}

\section{Lower semicontinuity of attractors}
\label{manifolds}

\ \ \ \ We are now in a position to prove our main result, the lower semicontinuity of attractors of the parabolic problem (\ref{probabstract}). This will follow from the continuity of the local unstable manifolds and the gradient structure of the flow.

We already know that if all equilibrium points of (\ref{2}) are hyperbolic, then there is only finite number of them, that is, $\mathscr{E}_{0}=\{u^{*}_{0,1},...,u^{*}_{0,m}\}$, and there exists $\epsilon_{0}>0$ such that, for each $0<\epsilon\leqslant \epsilon_{0}$, the set of equilibria of (\ref{1}) has exactly $m$ elements, say $\mathscr{E}_{\epsilon}=\{u^{*}_{\epsilon,1},...,u^{*}_{\epsilon,m}\}$, and $u^{*}_{\epsilon,i}\to u^{*}_{0,i}$ in $H^{1}(\Omega)$, as $\epsilon\to0$, for $i=1,...,m$, by Theorem~\ref{teo5eq}. For each $u^{*}_{\epsilon,i}\in \mathscr{E}_{\epsilon}$, with $\epsilon\in[0,\epsilon_{0}]$ and $i=1,...,m$, we define its unstable manifold
$$
\begin{array}{lll} W^{u}(u^{*}_{\epsilon,i})=\left\{\eta\in H^{1}(\Omega) \; : \; \mbox{there is a global solution $\xi:\mathbb{R} \mapsto H^{1}(\Omega)$ of (\ref{probabstract}) with $\xi(0)=\eta$}\right.\\
\qquad \qquad \qquad  \left.\mbox{such that $\xi(t)\to u^{*}_{\epsilon,i}$ in $H^{1}(\Omega)$, as $t\to -\infty$} \right\}, \end{array}
$$
and its $\delta$-local unstable manifolds as
$$
\begin{array}{lll} W^{u}_{\delta}(u^{*}_{\epsilon,i})=\left\{\eta\in B_{\delta}(u^{*}_{\epsilon,i})\subset H^{1}(\Omega) \; : \; \mbox{there is a global solution $\xi:\mathbb{R} \mapsto H^{1}(\Omega)$ of (\ref{probabstract}) with $\xi(0)=\eta$}\right.\\
\qquad \qquad \qquad \left. \mbox{such that $\xi(t)\in B_{\delta}(u^{*}_{\epsilon,i})$, $\forall$ $t\leqslant0$, and $\xi(t)\to u^{*}_{\epsilon,i}$ in $H^{1}(\Omega)$, as $t\to -\infty$}\right\}. \end{array}
$$
For further properties of local unstable manifolds, see~\cite{haleat}.

We will show that the local unstable manifolds of $u^{*}_{\epsilon,i}$, for each $i=1,...,m$ fixed, behave continuously with $\epsilon$ in $H^{1}(\Omega)$, using~\cite[Theorem 5.2]{antoniomarcone}, where abstract results on continuity of attractors were obtained. First, we need the following result:

\newtheorem{lem1var}{Lemma}[section]
\begin{lem1var}
\label{lem1var}
Suppose that (H) holds and let $\{u^{*}_{\epsilon}\}_{\epsilon\in(0,\epsilon_{0}]}\in H^{1}(\Omega)$ be a sequence of equilibria of (\ref{probabstract}) such that $u^{*}_{\epsilon}\to u^{*}_{0}$ in $H^{1}(\Omega)$, as $\epsilon\to0$, where $u^{*}_{0}\in H^{1}(\Omega)$ is an equilibrium point of (\ref{probabstract}) with $\epsilon=0$. 
Define the function $F:H^{1}(\Omega)\times [0,\epsilon_{0}]  \mapsto H^{-\alpha}(\Omega)$, with $\frac{1}{2}<\alpha<1$, by $F(u,\epsilon):=F_{\epsilon}(u)$, for $u\in H^{1}(\Omega)$ and $\epsilon\in[0,\epsilon_{0}]$, where $F_{\epsilon}$ is given in (\ref{defFe}). %
Then, $F$ is continuous, $F_\epsilon$ is $C^{1}$, the partial derivative $F_{u}(\cdot,\epsilon)$ is  continuous at $(u,0)$ for all $u\in H^{1}(\Omega)$, and $F(u^{*}_{\epsilon}+w,\epsilon)=A_{\epsilon}u^{*}_{\epsilon}+F_{u}(u^{*}_{\epsilon},\epsilon)w+r(w,\epsilon)$, for all $\epsilon\in[0,\epsilon_{0}]$. Also, for some $\rho > 0$
$$
\sup_{\left\|w\right\|_{H^{1}(\Omega)}\leqslant \rho} \left\|r(w,\epsilon)-r(w,0)\right\|_{H^{-\alpha}(\Omega)}\leqslant C(\epsilon),  \qquad \mbox{with $C(\epsilon)\to 0$ as $\epsilon\to0$,}
$$
$\left\|r(w_1,\epsilon)-r(w_2,\epsilon)\right\|_{H^{-\alpha}(\Omega)}\leqslant k(\rho)\left\|w_1-w_2\right\|_{H^{1}(\Omega)}$, for $\left\|w_{1}\right\|_{H^{1}(\Omega)}\leqslant \rho$, $\left\|w_{2}\right\|_{H^{1}(\Omega)}\leqslant \rho$, $k(\rho)\to0$ as $\rho\to0$, and $k(\cdot)$ is nondecreasing.
\end{lem1var}
\begin{proof}
From Lemma~\ref{lem1eq} we have that $F$ is continuous at $(u,0)$ for all $u\in H^{1}(\Omega)$. Also, the continuity of $F$ at $(u,\epsilon)$ for $\epsilon \neq0$ and $u\in H^{1}(\Omega)$ is immediate. 

Using Lemma~\ref{not1eq} we get that $F$ is Fr\'echet differentiable in the first variable $u$. Now, from Lemma~\ref{lem2eq} we have that the partial derivative $F_{u}$ is continuous at $(u,0)$ for $u\in H^{1}(\Omega)$. Moreover, since $F_{\epsilon}:H^{1}(\Omega) \mapsto H^{-\alpha}(\Omega)$ is Fr\'echet differentiable, uniformly in $\epsilon$, we have
$$
F_{\epsilon}(u^{*}_{\epsilon}+w)-F_{\epsilon}(u^{*}_{\epsilon})=F'_{\epsilon}(u^{*}_{\epsilon})w+r(w,\epsilon),
$$ 
with $r(0,\epsilon)=0$. Now, $u^{*}_{\epsilon}$ is an equilibrium of (\ref{probabstract}), that is, $A_{\epsilon}u^{*}_{\epsilon}=F_{\epsilon}(u^{*}_{\epsilon})$, hence
$$
F_{\epsilon}(u^{*}_{\epsilon}+w)=A_{\epsilon}u^{*}_{\epsilon}+F'_{\epsilon}(u^{*}_{\epsilon})w+r(w,\epsilon).
$$ 
For each $w\in H^{1}(\Omega)$ we have
\begin{eqnarray*}
&&\left\|r(w,\epsilon)-r(w,0)\right\|_{H^{-\alpha}(\Omega)} \\
&\leqslant& \left\|F_{\epsilon}(u^{*}_{\epsilon}+w)-F_{0}(u^{*}_{0}+w)\right\|_{H^{-\alpha}(\Omega)}+\left\|F_{\epsilon}(u^{*}_{\epsilon})-F_{0}(u^{*}_{0})\right\|_{H^{-\alpha}(\Omega)}+\left\|F'_{\epsilon}(u^{*}_{\epsilon})w-F'_{0}(u^{*}_{0})w\right\|_{H^{-\alpha}(\Omega)}.
\end{eqnarray*}

From Lemma~\ref{lem1eq} we have
$$
\left\|F_{\epsilon}(u^{*}_{\epsilon})-F_{0}(u^{*}_{0})\right\|_{H^{-\alpha}(\Omega)}\to 0, \qquad \mbox{as $\epsilon\to0$}.
$$
Moreover, due to Lemma~\ref{lem1eq}, we have 
$$
\begin{array}{rcl}
\left\|F_{\epsilon}(u^{*}_{\epsilon}+w)-F_{0}(u^{*}_{0}+w)\right\|_{H^{-\alpha}(\Omega)}&\leqslant& \left\|F_{\epsilon}(u^{*}_{\epsilon}+w)-F_{\epsilon}(u^{*}_{0}+w)\right\|_{H^{-\alpha}(\Omega)}+\left\|F_{\epsilon}(u^{*}_{0}+w)-F_{0}(u^{*}_{0}+w)\right\|_{H^{-\alpha}(\Omega)}\\
&\leqslant& L\left\|u^{*}_{\epsilon}-u^{*}_{0}\right\|_{H^{1}(\Omega)}+\left\|F_{\epsilon}(u^{*}_{0}+w)-F_{0}(u^{*}_{0}+w)\right\|_{H^{-\alpha}(\Omega)}\to 0,\qquad \mbox{as $\epsilon\to0$},
\end{array}
$$
uniformly for $w\in H^{1}(\Omega)$ such that $\left\|w\right\|_{H^{1}(\Omega)}\leqslant \rho$ for any $\rho >0$. Now, from Lemma~\ref{lem2eq} we have
$$
\left\|F'_{\epsilon}(u^{*}_{\epsilon})w-F'_{0}(u^{*}_{0})w\right\|_{H^{-\alpha}(\Omega)}\leqslant \left\|F'_{\epsilon}(u^{*}_{\epsilon})-F'_{0}(u^{*}_{0})\right\|_{\mathscr{L}(H^{1}(\Omega),H^{-\alpha}(\Omega))}\left\|w\right\|_{H^{1}(\Omega)}\to0,\qquad \mbox{as $\epsilon\to0$}.
$$
Therefore, $\sup_{\left\|w\right\|_{H^{1}(\Omega)}\leqslant \rho} \left\|r(w,\epsilon)-r(w,0)\right\|_{H^{-\alpha}(\Omega)}\leqslant C(\epsilon)$, with $C(\epsilon)\to0$ as $\epsilon\to0$.

Finally, let $w_1,w_2\in H^{1}(\Omega)$ such that $\left\|w_{1}\right\|_{H^{1}(\Omega)}\leqslant \rho$ and $\left\|w_{2}\right\|_{H^{1}(\Omega)}\leqslant \rho$. Using the Mean Value's Inequality and Lemma~\ref{lem2eq}, we have that there exist $0\leqslant s\leqslant 1$ and $L>0$ independent of $\epsilon$ such that
$$
\begin{array}{rcl}
\left\|r(w_1,\epsilon)-r(w_2,\epsilon)\right\|_{H^{-\alpha}(\Omega)}&=&\left\|F_{\epsilon}(u^{*}_{\epsilon}+w_1)-F_{\epsilon}(u^{*}_{\epsilon}+w_2)-F'_{\epsilon}(u^{*}_{\epsilon})w_1+F'_{\epsilon}(u^{*}_{\epsilon})w_2\right\|_{H^{-\alpha}(\Omega)}\\
&\leqslant& \left\|F'_{\epsilon}(s(u^{*}_{\epsilon}+w_1)+(1-s)(u^{*}_{\epsilon}+w_2))(w_{1}-w_{2})-F'_{\epsilon}(u^{*}_{\epsilon})(w_1-w_2)\right\|_{H^{-\alpha}(\Omega)}\\
&\leqslant& \left\|F'_{\epsilon}(u^{*}_{\epsilon}+sw_1+(1-s)w_2)-F'_{\epsilon}(u^{*}_{\epsilon})\right\|_{\mathscr{L}(H^{1}(\Omega),H^{-\alpha}(\Omega))}\left\|w_{1}-w_{2}\right\|_{H^{1}(\Omega)}\\
&\leqslant& L\left\|sw_1+(1-s)w_2\right\|_{H^{1}(\Omega)}\left\|w_{1}-w_{2}\right\|_{H^{1}(\Omega)}\leqslant 2L\rho \left\|w_{1}-w_{2}\right\|_{H^{1}(\Omega)}.
\end{array}
$$    
Taking $K(\rho)=2L\rho$ we obtain the results. 
\end{proof}

From Lemma~\ref{lem1var} and using~\cite[Theorem 5.2]{antoniomarcone}, we obtain the continuity of the local unstable manifolds near an equilibrium of (\ref{probabstract}). More precisely, we have:

\newtheorem{contvar}[lem1var]{Proposition}
\begin{contvar}
\label{contvar}
Suppose that (H) holds and that $u^{*}_{0}$ is a hyperbolic equilibrium point of (\ref{probabstract}) with $\epsilon=0$. By Theorem~\ref{teo4eq}, there exists $\epsilon_{0}>0$ such that (\ref{probabstract}) has an unique equilibrium $u^{*}_{\epsilon}\in H^{1}(\Omega)$ in a small neighborhood of $u^{*}_{0}$, for all $0<\epsilon\leqslant \epsilon_{0}$, with $u^{*}_{\epsilon}\to u^{*}_{0}$ in $H^{1}(\Omega)$, as $\epsilon\to0$. Then, there exists $\delta>0$ such that 
$$
\sup_{u_{\epsilon}\in W^{u}_{\delta}(u^{*}_{\epsilon})}\inf_{u_{0}\in W^{u}_{\delta}(u^{*}_{0})} \left\|u_{\epsilon}-u_{0}\right\|_{H^{1}(\Omega)}+ \sup_{u_{0}\in W^{u}_{\delta}(u^{*}_{0})}\inf_{u_{\epsilon}\in W^{u}_{\delta}(u^{*}_{\epsilon})} \left\|u_{\epsilon}-u_{0}\right\|_{H^{1}(\Omega)}\to 0,\qquad \mbox{as $\epsilon\to0$}.
$$
\end{contvar}

Now, we get the main result of this paper:

\newtheorem{inferior}[lem1var]{Theorem}
\begin{inferior}
\label{inferior}
Suppose that (H) holds and that every equilibria of (\ref{probabstract}) with $\epsilon=0$ is hyperbolic. Then, the family of global attractors of (\ref{probabstract}), $\{\mathscr{A}_{\epsilon}\}_{\epsilon\in[0,\epsilon_0]}$, is lower semicontinuous at $\epsilon=0$ in $H^{1}(\Omega)$.  
\end{inferior} 
\begin{proof}
Initially, we observe that the nonlinear semigroup $T^{0}(t)$ is a gradient system, $T^{\epsilon}(t)$ is asymptotically smooth and orbits of bounded sets are bounded, for any $\epsilon \in [0,\epsilon_0]$. Moreover, $T^{\epsilon}(t)\phi$ is continuous at $\epsilon=0$, uniformly with respect to $(t,\phi)$ in bounded sets of $\mathbb{R}^{+}\times H^{1}(\Omega)$, see Proposition~\ref{upperattrac}. 

Now, let us consider the operator $L_\epsilon(u^*) = A_{\epsilon}-F'_{\epsilon}(u^*)$ given by the linearization of \eqref{3e} at a hyperbolic equilibrium $u^* \in \mathscr{E}_{\epsilon}$. 
Due to Lemma \ref{lem2eq}, we can argue as in Lemma \ref{lem4} to prove that there exists $C(\epsilon) \geqslant 0$, $C(\epsilon) \to 0$ as $\epsilon \to 0$, such that 
\begin{eqnarray*}
\| \left( L_\epsilon(u^*) - L_0(u^*) \right) \, u \|_{H^{-\alpha}(\Omega)} 
 \leqslant  \hat C(\epsilon) \, \| A_0 u \|_{H^{-\alpha}(\Omega)} 
\leqslant C(\epsilon) \, \| L_0(u^*) u \|_{H^{-\alpha}(\Omega)} 
\end{eqnarray*}
for all $u \in H^{2-\alpha}(\Omega)$. Hence, using \cite[Theorem 3.3]{antoniomarcone}, we have that the resolvent operators ${L_\epsilon(u^*)}^{-1}$ converge to ${L_0(u^*)}^{-1}$ in operator norm.

As we have seen above, if all equilibrium points of (\ref{probabstract}) with $\epsilon=0$ are hyperbolic, then $\mathscr{E}_{0}=\{u^{*}_{0,1},...,u^{*}_{0,m}\}$ is finite and there exists $\epsilon_{0}>0$ such that $\mathscr{E}_{\epsilon}=\{u^{*}_{\epsilon,1},...,u^{*}_{\epsilon,m}\}$, for all $0<\epsilon\leqslant \epsilon_{0}$, with $u^{*}_{\epsilon,i}\to u^{*}_{0,i}$ in $H^{1}(\Omega)$, as $\epsilon\to0$, for all $i=1,...,m$, see Theorem~\ref{teo5eq}. 
Thus, for each $i=1,...,m$, we get from the convergence of $u^{*}_{\epsilon,i}$ to $u^{*}_{0,i}$ that ${L_\epsilon(u^{*}_{\epsilon,i})}^{-1}w_{\epsilon}$ converges to ${L_0(u^{*}_{0,i})}^{-1}w$ in $H^{1}(\Omega)$, as $\epsilon\to0$, whenever $w_{\epsilon}\to w$ in $H^{-\alpha}(\Omega)$, as $\epsilon\to0$.

Consequently, the hyperbolicity of $u^{*}_{0,i}$ implies the hyperbolicity of $u^{*}_{\epsilon,i}$, for all $i=1,...,m$ and $\epsilon\in(0,\epsilon_0]$, with $\epsilon_{0}$ sufficiently small. Moreover, from Proposition~\ref{contvar} we also have the continuity of the local unstable manifolds of $u^{*}_{\epsilon,i}$, for each $i=1,...,m$ fixed. 

With these considerations, the result follows from~\cite[Theorem 4.10.8]{haleat}. Note that it also is a consequence form \cite[Theorem 3.10]{antoniomarcone}.
\end{proof}

\bibliographystyle{plain}

\end{document}